\documentclass[a4paper,10pt]{article}
\usepackage{test,bm}
\usepackage[dvips]{hyperref}

\title{Weak Limit of the Geometric Sum of Independent But Not Identically Distributed Random Variables}
\author{Alexis Akira Toda\thanks{Department of Economics, Yale University, 28 Hillhouse Avenue, New Haven, CT 06511, USA. Email: alexisakira.toda@yale.edu.  This paper benefited from comments by Tomasz Kozubowski and Kenichiro Tanaka.  The financial supports from the Cowles Foundation, the Nakajima Foundation, and Yale University are greatly acknowledged.}}

\numberwithin{equation}{section}
\numberwithin{thm}{section}
\newcommand{\AL}{\mathcal{AL}}
\renewcommand{\L}{\mathcal{L}}
\renewcommand{\vec}[1]{\bm{#1}}
\renewcommand{\Re}{\operatorname{Re}}

\begin{document}
\maketitle

\begin{abstract}
We show that when $\set{X_j}$ is a sequence of independent (but not necessarily identically distributed) random variables which satisfies a condition similar to the Lindeberg condition, the properly normalized geometric sum $\sum_{j=1}^{\nu_p}X_j$ (where $\nu_p$ is a geometric random variable with mean $1/p$) converges in distribution to a Laplace distribution as $p\to 0$.  The same conclusion holds for the multivariate case.  This theorem provides a reason for the ubiquity of the double power law in economic and financial data.
\end{abstract}



\section{Introduction}
Let $\set{X_j}$ be a sequence of independent (but not necessarily identically distributed) random variables and $\nu_p$ be a geometric random variable with mean $1/p$ independent of $X_j$'s.  The geometric sum
\begin{equation}
\sum_{j=1}^{\nu_p}X_j\label{eq:intro.1}
\end{equation}
naturally arises in diverse fields \cite{kalashnikov1997}, particularly in economics.  For example, let $W_j$ be the financial wealth of a typical individual at age $j$ and suppose that the wealth grows in a multiplicative way according to $W_{j+1}=G_jW_j$, where $G_j$ is the growth rate which is a random variable.  Assuming that each individual dies with constant probability $p$ at each period (and a new individual is born), what does the cross-sectional distribution of wealth look like?  To answer this question, let $\nu_p$ be a geometric random variable that represents the age of the individual.  Letting $X_j=\log G_{j-1}$, the log wealth is
$$\log W_{\nu_p}=\log W_0+\sum_{j=0}^{\nu_p-1}\log G_j=\log W_0+\sum_{j=1}^{\nu_p}X_j,$$
a geometric sum.  It might be plausible to assume that $\set{X_j}$ is independent conditional on the realization of macro variables (GDP, interest rate, stock market returns, etc.), but since every individual is more or less affected by the state of the macroeconomy, it is not plausible to assume that $\set{X_j}$ (a time series) is identically distributed conditional on macro variables.  In that case the determination of the cross-sectional distribution of log wealth, or the weak limit of the geometric sum \eqref{eq:intro.1}, becomes a non-trivial problem.

The weak limit of the properly normalized random sum \eqref{eq:intro.1} (where $\nu_p$ is not necessarily a geometric random variable but a general integer-valued random variable) has been studied by a number of authors (see \cite{gnedenko-korolev1996} and the references therein).  In particular, when $\nu_p$ is a geometric random variable and $X_j$ has a finite variance, the weak limit of the properly normalized geometric sum \eqref{eq:intro.1} is a Laplace distribution \cite{kozubowski-rachev1999-multivariate,kozubowski-rachev1999-univariate}, which has been applied to modeling financial data \cite{mittnik-rachev1993,kozubowski-rachev1994,kozubowski-podgorski2001}.  However, the literature on the asymptotic distribution of geometric sums seems to be limited to the i.i.d.\ case.  \cite{szasz1972inid} and \cite{gnedenko-korolev1996} consider the asymptotic distribution of the random sum of independent but not identically distributed random variables and provide necessary and sufficient conditions for convergence, but since they do not provide explicit examples on geometric sums, it is not obvious whether their general theory applies to the specific case of geometric sums.  In this paper by using a technique similar to the proof of the Lindeberg-Feller central limit theorem, I show that the results for the geometric sum of i.i.d.\ random variables extend to the case when the random variables are independent but not identically distributed (i.n.i.d.).

Before proceeding to the main result we introduce some notations.  A random variable $X$ is said to be Laplace if it has a probability density function of the form
$$f(x)=\begin{cases}
\frac{\alpha\beta}{\alpha+\beta}\e^{-\alpha\abs{x-m}},&(x\ge m)\\
\frac{\alpha\beta}{\alpha+\beta}\e^{-\beta\abs{x-m}},&(x<m)
\end{cases}$$
where $m$ is the mode and $\alpha,\beta>0$ are shape parameters.  If $\alpha\neq \beta$, $X$ is said to be asymmetric Laplace.  The characteristic function of $X$ is
\begin{align*}
\phi_X(t)&=\int_{-\infty}^m\e^{itx}\frac{\alpha\beta}{\alpha+\beta}\e^{-\beta\abs{x-m}}\diff x+\int_m^\infty \e^{itx}\frac{\alpha\beta}{\alpha+\beta}\e^{-\alpha\abs{x-m}}\diff x\\
&=\frac{\e^{imt}}{1-i(\frac{1}{\alpha}-\frac{1}{\beta})t+\frac{t^2}{\alpha\beta}},
\end{align*}
from which we obtain the mean $m+\frac{1}{\alpha}-\frac{1}{\beta}$ and the variance $\frac{1}{\alpha^2}+\frac{1}{\beta^2}$.  It is often useful to parameterize the Laplace distribution in terms of its characteristic function.  Let $a=\frac{1}{\alpha}-\frac{1}{\beta}$ be an asymmetry parameter and $\sigma=\sqrt{\frac{2}{\alpha\beta}}$ be a scale parameter.  Then we write $X\sim \AL(m,a,\sigma)$ if
$$\phi_X(t)=\frac{\e^{imt}}{1-iat+\frac{\sigma^2t^2}{2}}.$$
The mean, mode, and variance of $\AL(m,a,\sigma)$ is $m+a$, $m$, and $a^2+\sigma^2$, respectively.  In particular, setting $\alpha=\beta=\frac{\sqrt{2}}{\sigma}$, the symmetric Laplace distribution with mean and mode $m$ and standard deviation $\sigma$ (which we denote by $\L(m,\sigma)$) has density $f(x)=\frac{1}{\sqrt{2}\sigma}\e^{-\frac{\sqrt{2}\abs{x-m}}{\sigma}}$ and characteristic function $\frac{\e^{imt}}{1+\frac{\sigma^2t^2}{2}}$.
A comprehensive review of the Laplace distribution can be found in \cite{kotz-kozubowski-podgorski2001}.

\section{Main result}\label{sec:main}
\begin{thm}\label{thm:1}
Let $\set{X_j}$ be a sequence of independent but not identically distributed (i.n.i.d) random variables such that $\E[X_j]=0$ and $\var[X_j]=\sigma_j^2$, $\set{a_j}$ be a real sequence, and $\nu_p$ be a geometric random variable independent of $X_j$'s with mean $1/p$.  Suppose that
\begin{enumerate}
\item $\lim\limits_{n\to\infty}n^{-\alpha}\sigma_n^2=0$ for some $0<\alpha<1$ and $\sigma^2:=\lim\limits_{n\to \infty}\frac{1}{n}\sum_{j=1}^n\sigma_j^2>0$ exists,
\item $a:=\lim\limits_{n\to\infty}\frac{1}{n}\sum_{j=1}^n a_j$ exists, and
\item for all $\epsilon>0$ we have
\begin{equation}
\lim_{p\to 0}\sum_{j=1}^\infty (1-p)^{j-1}p\E\left[X_j^2\set{\abs{X_j}\ge \epsilon p^{-\frac{1}{2}}}\right]=0.\label{eq:main.1}
\end{equation}
\end{enumerate}
Then, as $p\to 0$ the geometric sum $p^\frac{1}{2}\sum_{j=1}^{\nu_p}(X_j+p^\frac{1}{2}a_j)$ converges in distribution to $\AL(0,a,\sigma)$.
\end{thm}

By strengthening the assumptions of Theorem \ref{thm:1}, we obtain the following corollaries.

\begin{cor}\label{cor:main.2}
Let $\set{X_j}$ be a sequence of independent but not identically distributed (i.n.i.d) random variables such that $\E[X_j]=0$, $\var[X_j]=\sigma_j^2$, and $\sigma^2:=\lim\limits_{n\to \infty}\frac{1}{n}\sum_{j=1}^n\sigma_j^2>0$ exists.  Let $\set{a_j}$ be a real sequence such that $a:=\lim\limits_{n\to\infty}\frac{1}{n}\sum_{j=1}^n a_j$ exists, and $\nu_p$ be a geometric random variable independent of $X_j$'s with mean $1/p$.  Suppose that $\set{X_j^2}$ is uniformly integrable.  Then $p^\frac{1}{2}\sum_{j=1}^{\nu_p}(X_j+p^\frac{1}{2}a_j)\dto \AL(0,a,\sigma)$ as $p\to 0$.
\end{cor}

\begin{proof}
For $c>0$ let $M(c)=\sup_j \E[X_j^2\set{\abs{X_j}\ge c}]$.  Since $\set{X_j^2}$ is uniformly integrable, we have $M(c)\to 0$ as $c\to \infty$, so $M(c)<\infty$ for sufficiently large $c$.  For such $c$, we have
$$\sigma_j^2=\E[X_j^2\set{\abs{X_j}< c}]+\E[X_j^2\set{\abs{X_j}\ge c}]\le c^2+M(c),$$
so $\set{\sigma_j}$ is bounded, in particular $n^{-\alpha}\sigma_n^2\to 0$ for any $0<\alpha<1$.  For any $\epsilon>0$ and $c>0$ choose $p$ such that $\epsilon p^{-\frac{1}{2}}\ge c$.  Then
\begin{align*}
\sum_{j=1}^\infty (1-p)^{j-1}p\E\left[X_j^2\set{\abs{X_j}\ge \epsilon p^{-\frac{1}{2}}}\right]&\le \sum_{j=1}^\infty (1-p)^{j-1}p\E\left[X_j^2\set{\abs{X_j}\ge c}\right]\\
&\le M(c),
\end{align*}
so letting $p\to 0$ and then $c\to\infty$, condition \eqref{eq:main.1} holds.
\end{proof}

\begin{cor}\label{cor:main.3}
Let $\set{X_j}$ be a sequence of i.i.d.\ random variables with mean 0 and variance $\sigma^2$, $\set{a_j}$ a real sequence such that $a:=\lim\limits_{n\to\infty}\frac{1}{n}\sum_{j=1}^n a_j$ exists, and $\nu_p$ a geometric random variable independent of $X_j$'s with mean $1/p$.  Then $p^\frac{1}{2}\sum_{j=1}^{\nu_p}(X_j+p^\frac{1}{2}a_j)\dto \AL(0,a,\sigma)$ as $p\to 0$.
\end{cor}
\begin{proof}
Since $X_j$'s are i.i.d., $\lim\limits_{n\to \infty}\frac{1}{n}\sum_{j=1}^n\sigma_j^2 =\sigma^2>0$ and $\set{X_j^2}$ is uniformly integrable.  Hence the conclusion holds by Corollary \ref{cor:main.2}.
\end{proof}

The proof of Theorem \ref{thm:1} is based on the idea of Lindeberg \cite{lindeberg1922} for proving the central limit theorem.  We first prove Theorem \ref{thm:1} when $X_j$'s are Gaussian.  Then we take a sequence of independent zero mean Gaussian variables $\set{Y_j}$ with the same variances as $\set{X_j}$ and show that the geometric sums $p^\frac{1}{2}\sum_{j=1}^{\nu_p}(X_j+p^\frac{1}{2}a_j)$ and $p^\frac{1}{2}\sum_{j=1}^{\nu_p}(Y_j+p^\frac{1}{2}a_j)$ admit the same weak limit.

\begin{prop}\label{prop:2}
Let $\set{Y_j}$ be a sequence of independent Gaussian random variables such that $\E[Y_j]=0$, $\var[Y_j]=\sigma_j^2$, and $\sigma^2:=\lim\limits_{n\to \infty}\frac{1}{n}\sum_{j=1}^n\sigma_j^2>0$ exists.  Let $\set{a_j}$ be a real sequence such that $a:=\lim\limits_{n\to\infty}\frac{1}{n}\sum_{j=1}^n a_j$ exists, and $\nu_p$ be a geometric random variable independent of $Y_j$'s with mean $1/p$.  Then $p^\frac{1}{2}\sum_{j=1}^{\nu_p}(Y_j+p^\frac{1}{2}a_j)\dto \AL(0,a,\sigma)$ as $p\to 0$.
\end{prop}

\begin{proof}
Let $S_n=\sum_{j=1}^n Y_n$, $b_n=\sum_{j=1}^n a_n$, and $\tau_n^2=\sum_{j=1}^n \sigma_n^2$.  Since $Y_j$'s are independent Gaussian, $S_n\sim N(0,\tau_n^2)$.  By conditioning on $\nu_p$ the characteristic function of the geometric sum $Z_p:=p^\frac{1}{2}\sum_{j=1}^{\nu_p}(Y_j+p^\frac{1}{2}a_j)$ is
\begin{align*}
\phi_p(t)&:=\E[\e^{it Z_p}]=\E[\E[\e^{it Z_p}|\nu_p]]=\sum_{n=1}^\infty (1-p)^{n-1}p \E[\e^{it p^\frac{1}{2}(S_n+p^\frac{1}{2}b_n)} ]\\
&=\sum_{n=1}^\infty (1-p)^{n-1}p \e^{itpb_n-\frac{pt^2\tau_n^2}{2}}=\sum_{n=1}^\infty (1-p)^{n-1}p \e^{-pn (-it\frac{b_n}{n}+\frac{t^2\tau_n^2}{2n})}.
\end{align*}
Let $z_n=-it\frac{b_n}{n}+\frac{t^2\tau_n^2}{2n}$.  By assumption, $z_n\to z:=-iat+ \frac{\sigma^2t^2}{2}$ as $n\to\infty$.  Since $\Re z=\frac{\sigma^2t^2}{2}\ge 0>-1$, by Lemma \ref{lem:app.3}, we obtain
$$\lim_{p\to 0}\phi_p(t)=\frac{1}{1+z}=\frac{1}{1-iat+\frac{\sigma^2t^2}{2}}.$$
Hence $Z_p\dto \AL(0,a,\sigma)$ as $p\to 0$.
\end{proof}

Next we show that condition \eqref{eq:main.1} holds for $\set{Y_j}$.
\begin{lem}\label{lem:3}
Let everything be as in Theorem \ref{thm:1} and $\set{Y_j}$ be as in Proposition \ref{prop:2}.  Then condition \eqref{eq:main.1} holds for $\set{Y_j}$.
\end{lem}
\begin{proof}
Since $n^{-\alpha}\sigma_n^2\to 0$ for some $0<\alpha<1$, for any $\delta>0$ we can choose $N$ such that $n^{-\alpha}\sigma_n^2\le \delta$ for $n>N$.  Since by assumption $Y_j\sim N(0,\sigma_j^2)$, we have $Z=Y_j/\sigma_j\sim N(0,1)$.  Let $\eta=\frac{2\alpha}{1-\alpha}>0$ and $c=\epsilon p^{-\frac{1}{2}}$.  Since $\abs{Y_j/c}\ge 1$ when $\abs{Y_j}\ge c$, for $j>N$ we obtain
\begin{align}
\E\left[Y_j^2\set{\abs{Y_j}\ge c}\right]&\le \E\left[Y_j^2\abs{Y_j/c}^\eta\set{\abs{Y_j}\ge c}\right]\notag\\
&\le \E\left[Y_j^2\abs{Y_j/c}^\eta\right]\le \frac{\sigma_j^{2+\eta}}{c^\eta}\E[\abs{Z}^{2+\eta}]\notag\\
&\le \frac{\delta^\frac{2+\eta}{2} j^\frac{\alpha(2+\eta)}{2}}{c^\eta}\E[\abs{Z}^{2+\eta}]=\frac{\delta^\frac{1}{1-\alpha} j^\frac{\alpha}{1-\alpha}}{c^\eta}\E[\abs{Z}^{2+\eta}],\label{eq:main.2}
\end{align}
where we have used $\frac{2+\eta}{2}=\frac{1}{1-\alpha}$ by the definition of $\eta$. Substituting $c=\epsilon p^{-\frac{1}{2}}$, multiplying \eqref{eq:main.2} by $(1-p)^{j-1}p$ and summing over $j>N$, it follows from Lemma \ref{lem:app.4} that
\begin{align*}
&\sum_{j=N+1}^\infty (1-p)^{j-1}p\E\left[Y_j^2\set{\abs{Y_j}\ge \epsilon p^{-\frac{1}{2}}}\right]\\
&\le \frac{p}{1-p}\sum_{j=N+1}^\infty(1-p)^j \frac{p^\frac{\eta}{2}\delta^\frac{1}{1-\alpha} j^\frac{\alpha}{1-\alpha}}{\epsilon^\eta}\E[\abs{Z}^{2+\eta}]\\
&\le \frac{\delta^\frac{1}{1-\alpha}p^\frac{1}{1-\alpha}}{\epsilon^\eta(1-p)}\E[\abs{Z}^{2+\eta}]\sum_{j=1}^\infty (1-p)^j j^\frac{\alpha}{1-\alpha}\le \frac{C\delta^\frac{1}{1-\alpha}}{\epsilon^\eta(1-p)}\E[\abs{Z}^{2+\eta}]
\end{align*}
for some constant $C>0$.  Hence
\begin{multline*}
\sum_{j=1}^\infty (1-p)^{j-1}p\E\left[Y_j^2\set{\abs{Y_j}\ge \epsilon p^{-\frac{1}{2}}}\right]\\
\le \sum_{j=1}^N (1-p)^{j-1}p\E\left[Y_j^2\set{\abs{Y_j}\ge \epsilon p^{-\frac{1}{2}}}\right]+\frac{C\delta^\frac{1}{1-\alpha}}{\epsilon^\eta(1-p)}\E[\abs{Z}^{2+\eta}].
\end{multline*}
Letting $p\to 0$ and then $\delta \to 0$, condition \eqref{eq:main.1} holds for $\set{Y_j}$.
\end{proof}

\begin{prop}\label{prop:4}
Let $\set{X_j}$, $\set{Y_j}$, $\set{a_j}$, and $\nu_p$ be as in Theorem \ref{thm:1} and Proposition \ref{prop:2}. Then for any bounded $C^3$ function $f$ on $\R$ with bounded derivatives up to the third order, we have
$$\lim_{p\to 0}\abs{\E\left[f(p^\frac{1}{2}\textstyle\sum_{j=1}^{\nu_p}(X_j+p^\frac{1}{2}a_j))\right]
-\E\left[f(p^\frac{1}{2}\textstyle\sum_{j=1}^{\nu_p}(Y_j+p^\frac{1}{2}a_j))\right]}=0.$$

\end{prop}

\begin{proof}
Fix $n$ and consider
\begin{align}
&f(p^\frac{1}{2}\textstyle\sum_{j=1}^n(X_j+p^\frac{1}{2}a_j))
-f(p^\frac{1}{2}\sum_{j=1}^n(Y_j+p^\frac{1}{2}a_j))\notag\\
&=\sum_{j=1}^n[f(p^\frac{1}{2}(X_j+Z_j))-f(p^\frac{1}{2}(Y_j+Z_j))],\label{eq:main.3}
\end{align}
where
$$Z_j=X_1+\dotsb+X_{j-1}+Y_{j+1}+\dotsb+Y_n+p^\frac{1}{2}\sum_{k=1}^na_k.$$
By Corollary \ref{cor:app.6}, the $j$-th term of \eqref{eq:main.3} is equal to
\begin{multline}
f(p^\frac{1}{2}(X_j+Z_j))-f(p^\frac{1}{2}(Y_j+Z_j))\\
=f'(p^\frac{1}{2}Z_j)p^\frac{1}{2}(X_j-Y_j)+\frac{f''(p^\frac{1}{2}Z_j)}{2}p(X_j^2-Y_j^2)+R_j,\label{eq:main.4}
\end{multline}
where $R_j$ is the remainder term.  $R_j$ is bounded by
\begin{equation}
\abs{R_j}\le g(p^\frac{1}{2}X_j)+g(p^\frac{1}{2}Y_j),\label{eq:main.5}
\end{equation}
where $g(h)=K\min\set{h^2,\abs{h}^3}$ for some $K>0$.  Noting that $X_j,Y_j$ are independent of $Z_j$, $\E[X_j]=\E[Y_j]=0$, and $\var[X_j]=\var[Y_j]=\sigma_j^2$, taking expectations of both sides of \eqref{eq:main.4}, we get
$$\abs{\E[f(p^\frac{1}{2}(X_j+Z_j))]-\E[f(p^\frac{1}{2}(Y_j+Z_j))]}\le \E[\abs{R_j}].$$
Therefore by the triangle inequality and \eqref{eq:main.5} we obtain
\begin{align}
&\abs{\E[f(p^\frac{1}{2}\textstyle\sum_{j=1}^n(X_j+p^\frac{1}{2}a_j))]
-\E[f(p^\frac{1}{2}\sum_{j=1}^n(Y_j+p^\frac{1}{2}a_j))]}\notag\\
&\le \sum_{j=1}^nE[\abs{R_j}]\le \sum_{j=1}^n\left(\E[g(p^\frac{1}{2}X_j)]+\E[g(p^\frac{1}{2}Y_j)]\right).\label{eq:main.6}
\end{align}
Using the definition of $g$, we can bound $\E[g(p^\frac{1}{2}X_j)]$ as
\begin{align}
\E[g(p^\frac{1}{2}X_j)]&=\E\left[g(p^\frac{1}{2}X_j)\set{\abs{X_j}<\epsilon p^{-\frac{1}{2}}}\right]+\E\left[g(p^\frac{1}{2}X_j)\set{\abs{X_j}\ge \epsilon p^{-\frac{1}{2}}}\right]\notag\\
&\le K\E\left[\abs{p^\frac{1}{2}X_j}^3\set{\abs{X_j}< \epsilon p^{-\frac{1}{2}}}\right]+K\E\left[(p^\frac{1}{2}X_j)^2\set{\abs{X_j}\ge \epsilon p^{-\frac{1}{2}}}\right]\notag\\
&=K\E\left[(p^\frac{1}{2}X_j)^2\epsilon\set{\abs{X_j}< \epsilon p^{-\frac{1}{2}}}\right]+K\E\left[(p^\frac{1}{2}X_j)^2\set{\abs{X_j}\ge \epsilon p^{-\frac{1}{2}}}\right]\notag\\
&\le \epsilon Kp\sigma_j^2+Kp \E\left[X_j^2\set{\abs{X_j}\ge \epsilon p^{-\frac{1}{2}}}\right].\label{eq:main.7}
\end{align}
Now let $\tau_n^2=\sum_{j=1}^n\sigma_j^2$.  Since $\tau_n^2/n\to \sigma^2$ by assumption, $\set{\tau_n^2/n}$ is bounded by some $M>0$.  Then
\begin{equation}
\sum_{n=1}^\infty (1-p)^{n-1}p\sum_{j=1}^n \epsilon Kp\sigma_j^2\le \sum_{n=1}^\infty (1-p)^{n-1}p \epsilon Kp Mn=\epsilon KM.\label{eq:main.8}
\end{equation}
Also, by changing the order of summation we obtain
\begin{align}
&\sum_{n=1}^\infty(1-p)^{n-1}p\sum_{j=1}^n Kp\E\left[X_j^2\set{\abs{X_j}\ge \epsilon p^{-\frac{1}{2}}}\right]\notag\\
&=\sum_{j=1}^\infty \sum_{n=j}^\infty (1-p)^{n-1}p Kp\E\left[X_j^2\set{\abs{X_j}\ge \epsilon p^{-\frac{1}{2}}}\right]\notag\\
&=K \sum_{j=1}^\infty(1-p)^{j-1}p \E\left[X_j^2\set{\abs{X_j}\ge \epsilon p^{-\frac{1}{2}}}\right].\label{eq:main.9}
\end{align}
Combining \eqref{eq:main.7}, \eqref{eq:main.8}, and \eqref{eq:main.9}, we obtain
\begin{multline}
\sum_{n=1}^\infty(1-p)^{n-1}p\sum_{j=1}^n \E[g(p^\frac{1}{2}X_j)]\\
\le KM\epsilon+K\sum_{j=1}^\infty (1-p)^{j-1}p \E\left[X_j^2\set{\abs{X_j}\ge \epsilon p^{-\frac{1}{2}}}\right].\label{eq:main.10}
\end{multline}
The same inequality as \eqref{eq:main.10} holds when $\set{X_j}$ is replaced by $\set{Y_j}$.  Hence applying condition \eqref{eq:main.1} to \eqref{eq:main.10} and invoking Lemma \ref{lem:3}, it follows from \eqref{eq:main.6} and \eqref{eq:main.10} that
\begin{align*}
&\abs{\E\left[f(p^\frac{1}{2}\textstyle\sum_{j=1}^{\nu_p}(X_j+p^\frac{1}{2}a_j))\right]
-\E\left[f(p^\frac{1}{2}\textstyle\sum_{j=1}^{\nu_p}(Y_j+p^\frac{1}{2}a_j))\right]}\\
&\le \sum_{n=1}^\infty (1-p)^{n-1}p \abs{\E\left[f(p^\frac{1}{2}\textstyle\sum_{j=1}^n(X_j+p^\frac{1}{2}a_j))\right]
-\E\left[f(p^\frac{1}{2}\textstyle\sum_{j=1}^n(Y_j+p^\frac{1}{2}a_j))\right]}\\
&\le \sum_{n=1}^\infty (1-p)^{n-1}p \sum_{j=1}^n\left(\E[g(p^\frac{1}{2}X_j)]+\E[g(p^\frac{1}{2}Y_j)]\right)\\
&\le 2KM\epsilon+K\sum_{j=1}^\infty (1-p)^{j-1}p \left(\E\left[X_j^2\set{\abs{X_j}\ge \epsilon p^{-\frac{1}{2}}}\right]+\E\left[Y_j^2\set{\abs{Y_j}\ge \epsilon p^{-\frac{1}{2}}}\right]\right)\\
&\to 2KM\epsilon
\end{align*}
as $p\to 0$.  Since $\epsilon>0$ is arbitrary, letting $\epsilon\to 0$ we get
$$\lim_{p\to 0}\abs{\E\left[f(p^\frac{1}{2}\textstyle\sum_{j=1}^{\nu_p}(X_j+p^\frac{1}{2}a_j))\right]
-\E\left[f(p^\frac{1}{2}\textstyle\sum_{j=1}^{\nu_p}(Y_j+p^\frac{1}{2}a_j))\right]}=0. $$
\end{proof}

\begin{proof}[Proof of Theorem \ref{thm:1}]
Let $f(x)=\e^{itx}$.  $f$ is $C^\infty$ and all of its derivatives are bounded because $\abs{f^{(n)}(x)}=\abs{(it)^n\e^{itx}}=t^n$, which does not depend on $x$.  Let $\set{Y_j}$ be as in Proposition \ref{prop:2}.  Then by Proposition \ref{prop:4}, we get
$$\lim_{p\to 0}\abs{\E\left[\e^{itp^\frac{1}{2}\sum_{j=1}^{\nu_p}(X_j+p^\frac{1}{2}a_j)}\right]
-\E\left[\e^{itp^\frac{1}{2}\sum_{j=1}^{\nu_p}(Y_j+p^\frac{1}{2}a_j)}\right]}=0.$$
Hence by Proposition \ref{prop:2} we have
$$\lim_{p\to 0}\E\left[\e^{itp^\frac{1}{2}\sum_{j=1}^{\nu_p}(X_j+p^\frac{1}{2}a_j)}\right]=\lim_{p\to 0}\E\left[\e^{itp^\frac{1}{2}\sum_{j=1}^{\nu_p}(Y_j+p^\frac{1}{2}a_j)}\right]=\frac{1}{1-iat+\frac{\sigma^2t^2}{2}}.$$
Since the right-most expression is the characteristic function of $\AL(0,a,\sigma)$ which is continuous at $t=0$, by L\'evy's continuity theorem $p^\frac{1}{2}\sum_{j=1}^{\nu_p}(X_j+p^\frac{1}{2}a_j)$ converges in distribution to $\AL(0,a,\sigma)$ as $p\to 0$.
\end{proof}

\section{Multivariate case}\label{sec:mult}
The generalization of Theorem \ref{thm:1} to the multivariate case is straightforward.  If $\vec{X}$ is a $d$-dimensional random variable with characteristic function
$$\phi_{\vec{X}}(\vec{t})=\frac{\e^{i\vec{m}'\vec{t}}}{1-i\vec{a}'\vec{t}+\frac{1}{2}\vec{t}'\vec{\Sigma} \vec{t}},$$
where $\vec{m},\vec{a}\in\R^d$ and $\vec{\Sigma}$ is a $d\times d$ symmetric and positive definite matrix, then the distribution of $\vec{X}$ is said to be multivariate Laplace which we denote by $\AL_d(\vec{m},\vec{a},\vec{\Sigma})$.  The mean, mode, and variance of $\AL_d(\vec{m},\vec{a},\vec{\Sigma})$ is $\vec{m}+\vec{a}$, $\vec{m}$, and $\vec{\Sigma}+\vec{a}\vec{a}'$, respectively.

\begin{thm}
Let $\set{\vec{X}_j}$ be a sequence of independent but not identically distributed (i.n.i.d) random vectors in $\R^d$ such that $\E[\vec{X}_j]=\vec{0}$ and $\var[\vec{X}_j]=\vec{\Sigma}_j$, $\set{\vec{a}_j}$ be a sequence in $\R^d$, and $\nu_p$ be a geometric random variable independent of $\vec{X}_j$'s with mean $1/p$.  Suppose that
\begin{enumerate}
\item $\lim\limits_{n\to\infty}n^{-\alpha}\vec{\Sigma}_n=\vec{O}$ for some $0<\alpha<1$ and $\vec{\Sigma}:=\lim\limits_{n\to \infty}\frac{1}{n}\sum_{j=1}^n\vec{\Sigma}_j$ exists and positive definite,
\item $\vec{a}:=\lim\limits_{n\to\infty}\frac{1}{n}\sum_{j=1}^n \vec{a}_j$ exists, and
\item for all $\epsilon>0$ we have
\begin{equation}
\lim_{p\to 0}\sum_{j=1}^\infty (1-p)^{j-1}p\E\left[\norm{\vec{X}_j}^2\set{\norm{\vec{X}_j}\ge \epsilon p^{-\frac{1}{2}}}\right]=0,\label{eq:mult.1}
\end{equation}
where $\norm{\cdot}$ denotes the Euclidean norm.
\end{enumerate}
Then, as $p\to 0$ the geometric sum $p^\frac{1}{2}\sum_{j=1}^{\nu_p}(\vec{X}_j+p^\frac{1}{2}\vec{a}_j)$ converges in distribution to $\AL_d(\vec{0},\vec{a},\vec{\Sigma})$.
\end{thm}
\begin{proof}
Let us first show that for any $\vec{0}\neq \vec{t}\in\R^d$ the sequence of real random variables $\set{\vec{t}'\vec{X}_j}$ satisfies the assumptions of Theorem \ref{thm:1}.  Since $\E[\vec{X}_j]=\vec{0}$ and $\var[\vec{X}_j]=\vec{\Sigma}_j$, we have $\E[\vec{t}'\vec{X}_j]=0$ and $\var[\vec{t}'\vec{X}_j]=\vec{t}'\vec{\Sigma}_j\vec{t}$.  Hence $\lim\limits_{n\to\infty}n^{-\alpha}\var[\vec{t}'\vec{X}_n]=0$ and
$$\lim\limits_{n\to\infty}\frac{1}{n}\sum_{j=1}^n\var[\vec{t}'\vec{X}_j]=\vec{t}'\vec{\Sigma} \vec{t}>0$$
because $\vec{\Sigma}$ is positive definite and $\vec{t}\neq \vec{0}$.  Also, $\lim\limits_{n\to\infty}\frac{1}{n}\sum_{j=1}^n\vec{t}'\vec{a}_j=\vec{t}'\vec{a}$.  By the Cauchy-Schwarz inequality, we have $\abs{\vec{t}'\vec{X}}\le \norm{\vec{t}}\norm{\vec{X}}$.  Hence
$$\set{\abs{\vec{t}'\vec{X}}\ge c}\subset \set{\norm{\vec{t}}\norm{\vec{X}}\ge c}=\set{\norm{\vec{X}}\ge \frac{c}{\norm{\vec{t}}}}.$$
Therefore for all $\epsilon>0$ we have
\begin{align*}
&\sum_{j=1}^\infty (1-p)^{j-1}p \E\left[(\vec{t}'\vec{X}_j)^2\set{\abs{\vec{t}'\vec{X}_j}\ge \epsilon p^{-\frac{1}{2}}}\right]\\
&\le \norm{\vec{t}}^2\sum_{j=1}^\infty (1-p)^{j-1}p \E\left[\norm{\vec{X}_j}^2\set{\norm{\vec{X}_j}\ge \frac{\epsilon}{\norm{\vec{t}}} p^{-\frac{1}{2}}}\right]\to 0
\end{align*}
as $p\to 0$ by condition \eqref{eq:mult.1}, so $\set{\vec{t}'\vec{X}_j}$ satisfies condition \eqref{eq:main.1} of Theorem \ref{thm:1}.  Since $\set{\vec{t}'\vec{X}_j}$ satisfies all assumptions of Theorem \ref{thm:1}, it follows that
$$\vec{t}'p^\frac{1}{2}\sum_{j=1}^{\nu_p}(\vec{X}_j+p^\frac{1}{2}\vec{a}_j)
=p^\frac{1}{2}\sum_{j=1}^{\nu_p}(\vec{t}'\vec{X}_j+p^\frac{1}{2}\vec{t}'\vec{a}_j)\dto \AL(\vec{0},\vec{t}'\vec{a},\sqrt{\vec{t}'\vec{\Sigma}\vec{t}})$$
as $p\to 0$.  This shows that
\begin{equation}
\lim_{p\to 0}\E\left[\e^{i\vec{t}'p^\frac{1}{2}\sum_{j=1}^{\nu_p}(\vec{X}_j+p^\frac{1}{2}\vec{a}_j)}\right]=\frac{1}{1-i\vec{t}'\vec{a}+\frac{1}{2}\vec{t}'\vec{\Sigma}\vec{t}}.\label{eq:mult.2}
\end{equation}
In proving \eqref{eq:mult.2} we have assumed that $\vec{t}\neq \vec{0}$, but \eqref{eq:mult.2} trivially holds for $\vec{t}=\vec{0}$.  Since the right-most expression of \eqref{eq:mult.2} is continuous at $\vec{t}=\vec{0}$, by L\'evy's continuity theorem
$$p^\frac{1}{2}\sum_{j=1}^{\nu_p}(\vec{X}_j+p^\frac{1}{2}\vec{a}_j)\dto \AL_d(\vec{0},\vec{a},\vec{\Sigma}). $$
\end{proof}

\section{Concluding remarks}
In this paper I showed that the properly normalized geometric sum $\sum_{j=1}^{\nu_p}X_j$ converges in distribution to a Laplace random variable even if the random variables $\set{X_j}$ are not identically distributed as long as they are independent.  The proof is similar to that of the Lindeberg-Feller central limit theorem.  This theorem provides a reason why many economic and financial variables obey the power law not just in the right tail \cite{gabaix2009} but also in the left tail.  If an economic variable results from a large, deterministic number of independent multiplicative shocks, that variable will be lognormally distributed as first observed by \cite{gibrat1931}.  However, in reality many variables seem to be well-described by the double Pareto and related distribution  \cite{reed2001,reed2003,giesen-zimmermann-suedekum2010,toda-income-PRE,AMIK2012}.  If we incorporate the death probability of economic units in the model, the number of multiplicative shocks is not deterministic but a geometric random variable.  Theorem \ref{thm:1} (in exponential form) then states that the geometric product of independent positive random variables tends to the double Pareto distribution, which is empirically supported.

Since the central limit theorem holds under general conditions (for example, ergodicity and stationarity), we can expect that the properly normalized geometric sum of random variables converges in distribution to a Laplace distribution under such conditions even if independence fails.  Addressing these issues are beyond the scope of this paper but interesting to pursue. 

\appendix

\section{Lemmas}
\begin{lem}\label{lem:app.1}
For $z\in \C$, we have $\abs{\e^z-1}\le \abs{z}\e^{\abs{z}}$.
\end{lem}
\begin{proof}
Using the Taylor expansion of $\e^z$, we obtain
$$\abs{\e^z-1}=\abs{\sum_{n=1}^\infty \frac{z^n}{n!}}\le \abs{z}\sum_{n=0}^\infty \frac{\abs{z}^n}{(n+1)!}\le \abs{z}\sum_{n=0}^\infty \frac{\abs{z}^n}{n!}=\abs{z}\e^{\abs{z}}. $$
\end{proof}

\begin{lem}\label{lem:app.2}
For $0<p<1$ and $x\ge -1$, we have $0<(1-p)\e^{-px}<1$.
\end{lem}
\begin{proof}
Since $\e^t\ge 1+t$ for all $t$, we get $\e^{px}\ge 1+px\ge 1-p>0$.  The first equality holds if and only if $px=0$ and the second if and only if $x=-1$, but since $0<p<1$ the two equalities cannot hold simultaneously.  Hence $0<(1-p)\e^{-px}<1$.
\end{proof}

\begin{lem}\label{lem:app.3}
Let $0<p<1$ and $\set{z_n}\subset \C$ be such that $\lim\limits_{n\to \infty}z_n=z$ with $\Re z>-1$.  Then
$$\lim_{p\to 0}\sum_{n=1}^\infty (1-p)^{n-1}p\e^{-pn z_n}=\frac{1}{1+z}.$$
\end{lem}

\begin{proof}
For $0<p<1$ let $S(p)=\sum_{n=1}^\infty (1-p)^{n-1}p\e^{-pn z_n}$.  First we prove that $S(p)$ exists.  For this purpose let $z_n=x_n+iy_n$.  Since $\lim x_n>-1$, we can choose $N>0$ such that $x_n>-1$ for $n>N$.  Then by the triangle inequality
\begin{align*}
\abs{S(p)}&\le \sum_{n=1}^N (1-p)^{n-1}p\e^{-pn x_n}+\sum_{n=N+1}^\infty (1-p)^{n-1}p\e^{-pn x_n}\\
&\le \sum_{n=1}^N (1-p)^{n-1}p\e^{-pn x_n}+\sum_{n=N+1}^\infty (1-p)^{n-1}p\e^{pn}\\
&=\sum_{n=1}^N (1-p)^{n-1}p\e^{-pn x_n}+(1-p)^N \frac{p\e^p}{1-(1-p)\e^p}
\end{align*}
because $0<(1-p)\e^p<1$ by setting $x=-1$ in Lemma \ref{lem:app.2}.  Hence $S(p)$ exists.  Replacing $\set{z_n}$ with $z$ and applying the same argument,
$$T(p):=\sum_{n=1}^\infty (1-p)^{n-1}p\e^{-pn z}=\frac{p\e^{-pz}}{1-(1-p)\e^{-pz}}$$
exists.  Now by l'H\^opital's rule we have
$$\lim_{p\to 0}T(p)=\lim_{p\to 0}\frac{p}{\e^{pz}-(1-p)}=\lim_{p\to 0}\frac{1}{z\e^{pz}+1}=\frac{1}{1+z},$$
so it suffices to show that $\abs{S(p)-T(p)}\to 0$ as $p\to 0$.  For any $0<\epsilon<1+\Re z$, choose $N>0$ such that $\abs{z_n-z}<\epsilon$ for $n>N$.  Consider
\begin{multline*}
\abs{S(p)-T(p)}\le \sum_{n=1}^N (1-p)^{n-1}p\abs{\e^{-pnz_n}-\e^{-pnz}}\\
+\sum_{n=N+1}^\infty (1-p)^{n-1}p\abs{\e^{-pnz_n}-\e^{-pnz}}=I+II.
\end{multline*}
Since each term of $S(p)$ and $T(p)$ tends to zero as $p\to 0$, we have $I\to 0$.  By the choice of $N$ and Lemma \ref{lem:app.1}, letting $z=x+iy$ we get
\begin{align*}
II&\le \sum_{n=1}^\infty (1-p)^{n-1}p\e^{-pnx}\abs{\e^{-pn(z_n-z)}-1}\le \sum_{n=1}^\infty (1-p)^{n-1}p\e^{-pnx} pn\abs{z_n-z}\e^{pn\abs{z_n-z}}\\
&\le \sum_{n=N+1}^\infty (1-p)^{n-1}p\e^{-pnx} pn\epsilon\e^{pn\epsilon}=\sum_{n=N+1}^\infty (1-p)^{n-1}p\e^{-pn(x-\epsilon)}pn\epsilon.
\end{align*}
Since $\Re z-\epsilon>-1$, by Lemma \ref{lem:app.2} we have $(1-p)\e^{-p(x-\epsilon)}<1$, so the above sum converges.  Then
\begin{align*}
II&\le \epsilon p^2\sum_{n=N+1}^\infty (1-p)^{n-1}\e^{-pn(x-\epsilon)}n\\
&=\epsilon\left(\frac{p}{\e^{p(x-\epsilon)}-1+p}\right)^2 (1-p)^N\e^{-p(N-1)(z-\epsilon)}\to \frac{\epsilon}{(1+x-\epsilon)^2}
\end{align*}
as $p\to 0$ by applying l'H\^opital's rule to the above fraction.  Letting $\epsilon\to 0$, we obtain $II\to 0$.  Hence $\lim S(p)=\lim T(p)=\frac{1}{1+z}$.
\end{proof}

\begin{lem}\label{lem:app.4}
If $0<p<1$ and $\alpha>-1$, there exists a constant $C>0$ such that
$$\sum_{n=1}^\infty (1-p)^n n^\alpha \le Cp^{-\alpha-1}.$$
\end{lem}
\begin{proof}
Let $f(x)=(1-p)^x x^\alpha$.  If $-1<\alpha\le 0$, $f(x)$ is monotone decreasing for $x\ge 0$.  Since $\log(1-p)\le -p$, we obtain
$$\sum_{n=1}^\infty (1-p)^n n^\alpha\le \int_0^\infty (1-p)^x x^\alpha\diff x=\frac{\Gamma(\alpha+1)}{(-\log(1-p))^{\alpha+1}}\le Cp^{-\alpha-1},$$
where $C=\Gamma(\alpha+1)$.  If $\alpha>0$, again using $\log(1-p)\le -p$ we have
$$f'(x)=[\alpha+\log(1-p)x](1-p)^x x^{\alpha-1}\le 0$$
for $x\ge \frac{\alpha}{p}$.  Hence we obtain
\begin{align*}
\sum_{n=1}^\infty (1-p)^n n^\alpha&=\sum_{n\le \alpha/p}^\infty (1-p)^n n^\alpha+\sum_{n> \alpha/p}^\infty (1-p)^n n^\alpha\\
&\le \left(\frac{\alpha}{p}\right)\left(\frac{\alpha}{p}\right)^\alpha+\int_0^\infty (1-p)^x x^\alpha\diff x\\
&=\left(\frac{\alpha}{p}\right)^{\alpha+1}+\frac{\Gamma(\alpha+1)}{(-\log(1-p))^{\alpha+1}}\le Cp^{-\alpha-1},
\end{align*}
where $C=\alpha^{\alpha+1}+\Gamma(\alpha+1)$.
\end{proof}

\begin{lem}\label{lem:app.5}
Let $f$ be a bounded $C^3$ function on $\R$ with bounded derivatives up to the third order.  Then there exists a constant $K$ such that for all $h\in\R$, we have
$$g(h):=\sup_{x\in \R}\abs{f(x+h)-f(x)-f'(x)h-\frac{1}{2}f''(x)h^2}\le K\min\set{h^2,\abs{h}^3}.$$
\end{lem}
\begin{proof}
By assumption $M_i:=\sup_{x\in\R}\abs{f^{(i)}(x)}<\infty$ for $i=0,1,2,3$.  By Taylor's theorem for each $x,x+h\in \R$ there exists $\xi$ between $x$ and $x+h$ such that
$$f(x+h)=f(x)+f'(x)h+\frac{1}{2}f''(x)h^2+\frac{f'''(\xi)}{6}h^3.$$
Hence
$$\abs{f(x+h)-f(x)-f'(x)h-\frac{1}{2}f''(x)h^2}\le \frac{M_3}{6}\abs{h}^3=:K_1\abs{h}^3.$$
On the other hand, by the triangle inequality we get
$$\abs{f(x+h)-f(x)-f'(x)h-\frac{1}{2}f''(x)h^2}\le 2M_0+M_1\abs{h}+\frac{M_2}{2}\abs{h}^2\le K_2\abs{h}^2$$
for large enough $\abs{h}$, say $\abs{h}\ge b$.  Then for $\abs{h}<b$ we have
$$\abs{f(x+h)-f(x)-f'(x)h-\frac{1}{2}f''(x)h^2}\le K_1\abs{h}^3\le K_1b\abs{h}^2.$$
Hence by taking $K=\max\set{K_1,K_2,K_1b}$ we obtain
$$\abs{f(x+h)-f(x)-f'(x)h-\frac{1}{2}f''(x)h^2}\le K\min\set{h^2,\abs{h}^3}. $$
\end{proof}

\begin{cor}\label{cor:app.6}
Let everything be as in Lemma \ref{lem:app.5}.  Then
$$f(x+h_1)-f(x+h_2)=f'(x)(h_1-h_2)+\frac{f''(x)}{2}(h_1^2-h_2^2)+R(x,h_1,h_2),$$
where the remainder term $R$ satisfies
$$\abs{R(x,h_1,h_2)}\le g(h_1)+g(h_2)\le K\left[\min\set{h_1^2,\abs{h_1}^3}+\min\set{h_2^2,\abs{h_2}^3}\right]$$
\end{cor}
\begin{proof}
Trivial by Lemma \ref{lem:app.5} and the triangle inequality.
\end{proof} 

\bibliographystyle{plain}

\begin{thebibliography}{10}

\bibitem{AMIK2012}
Simone Alfarano, Mishael Milakovi{\'c}, Albrecht Irle, and Jonas Kauschke.
\newblock A statistical equilibrium model of competitive firms.
\newblock {\em Journal of Economic Dynamics and Control}, 36(1):136--149, 2012.

\bibitem{gabaix2009}
Xavier Gabaix.
\newblock Power laws in economics and finance.
\newblock {\em Annual Review of Economics}, 1:255--293, 2009.

\bibitem{gibrat1931}
Robert Gibrat.
\newblock {\em Les In{\'e}galit{\'e}s {\'e}conomiques}.
\newblock Librairie du Recueil Sirey, Paris, 1931.

\bibitem{giesen-zimmermann-suedekum2010}
Kristian Giesen, Arndt Zimmermann, and Jens Suedekum.
\newblock The size distribution across all cities---double {P}areto lognormal
  strikes.
\newblock {\em Journal of Urban Economics}, 68:129--137, 2010.

\bibitem{gnedenko-korolev1996}
Boris~V. Gnedenko and Victor~Yu. Korolev.
\newblock {\em Random Summation: Limit Theorems and Applications}.
\newblock CRC Press, Boca Raton, FL, 1996.

\bibitem{kalashnikov1997}
Vladimir~V. Kalashnikov.
\newblock {\em Geometric Sums: Bounds for Rare Events with Applications}.
\newblock Mathematics and Its Applications. Kluwer Academic Publishers,
  Dordrecht, The Netherlands, 1997.

\bibitem{kotz-kozubowski-podgorski2001}
Samuel Kotz, Tomasz~J. Kozubowski, and Krzysztof Podg{\'o}rski.
\newblock {\em The Laplace Distribution and Generalizations}.
\newblock Birkh{\"a}user, Boston, 2001.

\bibitem{kozubowski-podgorski2001}
Tomasz~J. Kozubowski and Krzysztof Podg\'orski.
\newblock Asymmetric {L}aplace laws and modeling financial data.
\newblock {\em Mathematical and Computer Modelling}, 34:1003--1021, 2001.

\bibitem{kozubowski-rachev1994}
Tomasz~J. Kozubowski and Svetlozar~T. Rachev.
\newblock The theory of geometric stable distributions and its use in modeling
  financial data.
\newblock {\em European Journal of Operational Research}, 74:310--324, 1994.

\bibitem{kozubowski-rachev1999-multivariate}
Tomasz~J. Kozubowski and Svetlozar~T. Rachev.
\newblock Multivariate geometric stable laws.
\newblock {\em Journal of Computational Analysis and Applications},
  1(4):349--385, 1999.

\bibitem{kozubowski-rachev1999-univariate}
Tomasz~J. Kozubowski and Svetlozar~T. Rachev.
\newblock Univariate geometric stable laws.
\newblock {\em Journal of Computational Analysis and Applications},
  1(2):177--217, 1999.

\bibitem{lindeberg1922}
Jarl~Waldemar Lindeberg.
\newblock {E}ine neue {H}erleitung des {E}xponentialgesetzes in der
  {W}ahrscheinlichkeitsrechnung.
\newblock {\em Mathematische Zeitschrift}, 15(1):211--225, 1922.

\bibitem{mittnik-rachev1993}
Stefan Mittnik and Svetlozar~T. Rachev.
\newblock Modeling asset returns with alternative stable distributions.
\newblock {\em Econometric Review}, 12(3):261--330, 1993.

\bibitem{reed2001}
William~J. Reed.
\newblock The {P}areto, {Z}ipf and other power laws.
\newblock {\em Economics Letters}, 74:15--19, 2001.

\bibitem{reed2003}
William~J. Reed.
\newblock The {P}areto law of incomes---an explanation and an extension.
\newblock {\em Physica A}, 319:469--486, 2003.

\bibitem{szasz1972inid}
Domokos Sz{\'a}sz.
\newblock Limit theorems for the distributions of the sums of a random number
  of random variables.
\newblock {\em Annals of Mathematical Statistics}, 43(6):1902--1913, 1972.

\bibitem{toda-income-PRE}
Alexis~Akira Toda.
\newblock Income dynamics with a stationary double {P}areto distribution.
\newblock {\em Physical Review E}, 83(4):046122, 2011.

\end{thebibliography}

\end{document}